\documentclass[12pt]{article}
\usepackage{color}

\usepackage{amsmath}
\usepackage{bm}
\usepackage{amssymb}
\usepackage{amsthm}
\usepackage{dsfont}
\usepackage{enumitem}
\usepackage{mathrsfs}
\usepackage{stmaryrd}
\usepackage[all]{xy}
\usepackage[mathcal]{eucal}
\usepackage{verbatim}  
\usepackage{fullpage}  
\usepackage{mdframed}
\usepackage{textcomp}

\renewcommand{\emptyset}{\O}

\usepackage[symbol]{footmisc}

\usepackage{tcolorbox}

\newtheorem{theorem}{Theorem}
\newtheorem{proposition}[theorem]{Proposition}
\newtheorem{lemma}[theorem]{Lemma}
\newtheorem{corollary}[theorem]{Corollary}

\newtheorem{remark}[theorem]{Remark}

\newtheorem{definition}[theorem]{Definition}
\newtheorem*{definition*}{Definition}

\newcommand{\Z}{\mathbb{Z}}
\newcommand{\R}{\mathbb{R}}
\newcommand{\Q}{\mathbb{Q}}

\numberwithin{equation}{subsection}

\usepackage{multicol}

\newcommand{\N}{\mathbb{N}}

\newcommand{\U}{\mathcal{U}}

\begin{document}
\title{Strictly Monotone Numerosity on Tame Sets via the Steiner Polynomial}
\author{Joseph T. Previdi}
\date{January 25, 2026}
\maketitle
\begin{abstract}
This paper uses inspiration from Integral Geometry to connect Tame Geometry with Nonstandard Analysis. We omit binomial coefficients from the Steiner polynomial to define the \textit{intrinsic volume polynomial} $\Phi$, a valuation defined on bounded definable sets in an o-minimal structure. We prove that using this normalization gives a strictly monotone valuation on point sets when the codomain $\R[t]$ is interpreted with ordering by end behavior. This leads to an algebraic version of Hadwiger's Theorem: $\Phi$ is the unique conormal continuous, similarity-equivariant homomorphism of ordered rings from $\mathcal{C}(\R^\infty) \to \mathbb{R}[t]$ (up to scaling). Noting that strict monotonicity is mirrored in numerosity theory (a branch of nonstandard analysis), we prove existence for a numerosity that exceptionally approximates the intrinsic volume polynomial. This suggests a connection between disparate fields, allowing each to complement the other.  
\end{abstract}

\section{Introduction}

How do we measure the size of a point set? In classical analysis, the Lebesgue measure provides a robust answer, yet is blind to lower-dimensional features: any lower-dimensional set is collapsed to zero. This is by design: a measure must lose some geometric information to gain countable additivity.

There exist a variety of methods to capture this lost information. Measure theory uses \textit{Hausdorff measure}; o-minimal and convex geometry use \textit{intrinsic volumes} (or quermassintegrals), such as the Euler characteristic, perimeter, and volume. These distinct functionals interact via many complex formulas.  Among these are the intrinsic volume product formula and the Steiner formula (see \cite{schneider_convex_2014}), which both show that the intrinsic volumes behave like coefficients of a polynomial.

In this paper, we propose unifying these invariants into a single polynomial in the style of Steiner, but with slight renormalization. We prove that this formalism has the tangible benefits of preserving the strict partial order on point sets by inclusion. Denote the \textbf{intrinsic volume polynomial} by $\Phi(A)(t) = \sum \mu_k(A)t^k$, where $\mu_k$ is the $k$-th intrinsic volume. If we equip the ring of polynomials with an ordering by end behavior ($t\to+\infty$), the immediate consequence is that $\Phi$ strictly respects the partial order of set inclusion, in the sense of Euclid's notion that ``the whole is greater than the part." That is, if $A\subsetneq B$, then $\Phi(A)<\Phi(B)$. 

This leads us to a corollary of Hadwiger's Characterization Theorem: any polynomial-valued valuation $\Psi$ on tame sets satisfying specific symmetries and continuity must be strictly monotone. This fact can be stated algebraically: when constructible functions are viewed as an ordered ring, the intrinsic volume polynomial $\Phi$ is the unique continuous, similarity-equivariant homomorphism of ordered rings into $\mathbb{R}[t]$ (up to scaling the variable $t$). 

Since applying Hadwiger to the polynomial form gives us strict monotonicity, we note immediately a relation to the Nonstandard Analysis theory of \textit{numerosities}, introduced by Benci and Di Nasso \cite{benci_elementary_2014}. This theory uses hyperreal-valued valuations to enforce strict monotonicity by treating valuations as hyperfinite counting functions. We demonstrate that $\Phi$ can be suitably modeled as a numerosity, up to infinitesimal approximation, allowing us to extend our domain of interest to \textit{all} point sets, instead of just tame ones.  

Specifically, Section 5 proves existence of a hyperfinite sample in a nonstandard extension of Euclidean space. Such a sample allows us to extend the intrinsic volume polynomial to a numerosity function whose domain is the full power set $\mathcal{P}(\R^d)$. We show that there exists a hyperinteger $\omega$ such that the numerosity $\mathfrak{n}(A)$ is infinitely close to the value $\Phi(A)(\omega)$. 

The results proved below tighten the relationship between valuations (or measures) and infinite hyperreal (or very large finite) counting. If the intrinsic volume polynomial $\Phi$ is (up to scaling) the unique polynomial satisfying geometrical symmetries and it is realizable in a nonstandard universe, then nonstandard analysis can provide a rigorous model for intrinsic geometry. This paper shows a concrete realization of that model.

\section{O-Minimal Structure and Intrinsic Volumes}\label{O-Minimal Geometry}
We work within a fixed o-minimal structure $\mathcal{S}=(\mathcal S_n)_{n\in \mathbb{N}}$ on the real line $\mathbb{R}$, following the foundational framework described by van den Dries in 1998 \cite{van_den_dries_tame_1998}. We call a set $A$ \textit{definable} or \textit{tame} if $A\in \mathcal{S}_n$ for some $n$. The o-minimal property of the structure ensures that $\mathcal S_1$ consists exactly of all finite unions of points and intervals. For a detailed treatment of o-minimal theory, see \cite{van_den_dries_tame_1998}. 

To ensure the intrinsic volumes are finite, we restrict our attention to the subclass of bounded definable sets within the structure. We denote this by $\mathcal{U}^b = \bigcup_{n\in\mathbb{N}}\mathcal{S}_n^b$, where $\mathcal S_n^b\subset \mathcal S_n$ refers to those definable sets of $\mathbb R^d$ that are bounded. Note that $\mathcal{U}^b$ does not constitute an o-minimal structure in the sense of van den Dries; while it is closed under unions, intersections and relative complements, it is not closed under absolute complements (e.g. $\mathbb R \setminus A$ is unbounded). The structure is also closed under Cartesian products, but does not include $\mathbb R$ itself or any Cartesian products thereof. 

The intrinsic volumes $\mu_0,\mu_1,..., \mu_d$ can be extended to o-minimal geometry from their use in the Steiner formula and other areas of convex geometry. For the o-minimal context, we define them as in Definition 2.2 from Wright \cite{wright_hadwiger_2011}. With this definition, the intrinsic volumes are independent of the ambient space, definable on all tame sets, and normalized such that $\mu_d$ is the $d$-volume and $\mu_0$ is the o-minimal Euler characteristic. 

We use the intrinsic volumes as coefficients for a polynomial:

\begin{definition}[Intrinsic Volume Polynomial]
For a bounded definable set $A\in \U^b$, define:
\begin{equation}
    \Phi(A)(t) = \sum_{k=0}^{\dim(A)} \mu_k(A) t^k.
\end{equation}
\end{definition}

Note that this is identical to the Steiner polynomial, except that the intrinsic volume polynomial does not include the factor of $\binom{d}{k}$. We omit the binomial coefficients to ensure the multiplicative property in Theorem 2 holds. 

\begin{theorem}
The intrinsic volume polynomial $\Phi: \mathcal{U}^b \to \mathbb{R}[t]$ is a valuation (in the sense of the inclusion-exclusion property) and preserves the Cartesian product. That is, for any $A,B \in \mathcal{U}^b$:
\begin{enumerate}
    \item \textbf{Valuation:} $\Phi(A \cup B) = \Phi(A) + \Phi(B) - \Phi (A\cap B)$
    \item \textbf{Multiplicative:} $\Phi(A \times B) = \Phi(A) \cdot \Phi(B)$
\end{enumerate}
\end{theorem}

\begin{proof}
The additivity (1) follows immediately from the valuation property of the intrinsic volumes $\mu_k$.

For multiplicativity (2), the product theorem (Theorem 4.6 in \cite{wright_hadwiger_2011}) suffices. The intrinsic volumes of a Cartesian product satisfy the  identity:
$$ \mu_k(A \times B) = \sum_{i+j=k} \mu_i(A) \mu_j(B). $$
Substituting this into the intrinsic volume polynomial (with $\dim A= n$ and $\dim B=m$):
$$ \Phi(A\times B) = \sum_{k=0}^{n+m} \left( \sum_{i+j=k} \mu_i(A) \mu_j(B) \right) t^k. $$
This is precisely the definition of the product of the polynomials $\Phi(A)$ and $\Phi(B)$. Thus:
$$ \Phi(A\times B) = \left( \sum_{i=0}^n \mu_i(A) t^i \right) \cdot \left( \sum_{j=0}^m \mu_j(B) t^j \right) = \Phi(A) \cdot \Phi(B). $$
Since $\mathcal{U}^b$ is graded by dimension in the structure $\mathcal S = (\mathcal S_n)_{n\in \mathbb{N}}$ and $\mathbb{R}[t]$ is graded by degree, and $\Phi$ maps an element of dimension $k$ to a polynomial of degree $k$, $\Phi$ preserves the graded structure.
\end{proof}

\section{Ordering and Strict Monotonicity}
The ultimate goal of defining the intrinsic volume polynomial is to allow for symmetrical size comparisons compatible with ``the whole is greater than the part" for point sets. Consequently, we also wish to compare the magnitudes of polynomials by end behavior. We interpret the polynomial ring $\mathbb{R}[t]$ as an ordered ring equipped with the standard lexicographical ordering induced by the degree.

\begin{definition}[{Lexicographical Ordering on $\mathbb{R}[t]$}]
Let $P(t) = a_n t^n + \dots + a_0 \in \mathbb{R}[t]$ be a non-zero polynomial. Let $k = \deg(P)$ be the largest index such that $a_k \neq 0$. We say that $P(t) > 0$ if the leading coefficient $a_k > 0$. We define $P > Q$ if $P - Q > 0$.
\end{definition}

This ordering is equivalent to comparing polynomials by end behavior, or as we will do later, evaluating them at a positive infinite hyperreal $\omega$. Lexicographical order by degree reflects the intuition that a higher-dimensional measure dominates any amount of lower-dimensional measure.

We show next that our repackaging, the intrinsic volume polynomial $\Phi$, is strictly monotone ($A\subsetneq B\implies\Phi(A)<\Phi(B)$), in contrast with the Lebesgue measure. 

\begin{theorem}[Strict Monotonicity]\label{Strict Monotonicity}
The map $\Phi$ is a strictly monotone valuation on definable sets. That is, for any $A, B \in \mathcal{U}^b$ with $A \subsetneq B$:
\begin{equation}
    \Phi(A) < \Phi(B)
\end{equation}
in the lexicographical order of $\mathbb{R}[t]$.
\end{theorem}

\begin{proof}
Let $A, B \in \mathcal{U}^b$ such that $A \subsetneq B$. Applying the additivity of $\Phi$ on disjoint sets:
$$ \Phi(B) = \Phi(A) + \Phi(B \setminus A). $$
Thus, $\Phi(B) - \Phi(A) = \Phi(B \setminus A)$.
Let $e=\dim(B\setminus A)$. Since $B\setminus A$ is nonempty and definable, by the Cell Decomposition Theorem \cite{van_den_dries_tame_1998} it contains an $e$-cell and thus has positive Hausdorff $e$-measure. Hence, $\Phi(B \setminus A) > 0$, since that positive measure is its leading coefficient. Therefore, $\Phi(A) < \Phi(B)$.
\end{proof}

\section{Monotonicity and Characterization}

We now turn to the fundamental characterization of the map $\Phi$. We have shown that $\Phi$ is a strictly monotone multiplicative valuation. We now prove that $\Phi$ is the unique map satisfying the natural geometric symmetries of bounded definable sets, up to a scaling of the variable $t$. This result is an analog of Hadwiger's Characterization Theorem.

\subsection{Geometric Machinery}

First, we formalize the notion of similarity equivariance in the context of polynomial-valued valuations. This is effectively the statement that the polynomial transforms identically to underlying set under dilations and rigid motions. Note a dilation (or scaling) by a positive real scale factor $\alpha$ is defined as $\alpha A = \{\alpha x \mid x \in A\}$. We call any combination of dilations and rigid motions $g$ a ``similarity transformation," and note that $g$ forms a group with a natural group action on point sets. 

\begin{definition}[Similarity Equivariance]
Let $\Psi: \mathcal{U}^b \to \mathbb{R}[t]$ be a map. We say that $\Psi$ is \textbf{similarity equivariant} if for any definable set $A \in \mathcal{U}^b$ and any similarity transformation $g$ with scale factor $\alpha > 0$:
\begin{equation}
    \Psi(g\cdot A)(t) = \Psi(A)(\alpha t).
\end{equation}
\end{definition}

This condition ensures that the variable $t$ scales as a length, $t^2$ as an area, and so on. It thus accurately reflects scaling in real space. Note that our intrinsic volume polynomial $\Phi$ satisfies this by construction, since intrinsic volumes are homogeneous of degree $k$:
$$ \Phi(g\cdot A)(t) = \sum_{k=0}^{\dim(A)} \mu_k(g\cdot A) t^k = \sum_{k=0}^{\dim(A)} (\alpha^k \mu_k(A)) t^k = \sum_{k=0}^{\dim(A)} \mu_k(A) (\alpha t)^k = \Phi(A)(\alpha t). $$

Our next hurdle is to demonstrate that the intrinsic volume polynomial $\Phi$ is, up to a scaling of the variable, the unique map satisfying the natural symmetries of the structure. We first establish a lemma regarding the behavior of equivariant polynomial maps on the unit interval $[0,1)$.

\begin{lemma}[The Scale Factor]
Let $\Psi: \mathcal{U}^b \to \mathbb{R}[t]$ be a similarity-equivariant valuation. Then for the unit interval $I = [0,1)$, we have:
$$ \Psi(I)(t) = \alpha t $$
for some constant $\alpha \in \mathbb{R}$, which we call the \textbf{scale factor} of $\Psi$.
\end{lemma}
\begin{proof}
Set $P(t) = \Psi([0,1))$. Additivity and invariance yield $\Psi([0,2)) = 2P(t)$, while scaling requires $\Psi([0,2)) = P(2t)$. The only polynomial solution to the functional equation $P(2t) = 2P(t)$ is linear, thus $\Psi([0,1))(t) = \alpha t$.
\end{proof}

The final ingredient before we prove uniqueness of the intrinsic volume polynomial is a lemma showing that the standard Hadwiger Theorem for convex bodies also applies in the case of definable sets in $\mathcal U^b$. It is a corollary of \cite{baryshnikov_hadwigers_2013}. 

\begin{lemma}[Hadwiger's Theorem for definable sets] Let $v$ be a valuation on  $\mathcal U^b$ which is conormal continuous (see \cite{baryshnikov_hadwigers_2013}) and Euclidean-invariant. Then $v$ is a linear combination of the intrinsic volumes. 
\end{lemma}
\begin{proof}
This follows directly from the classification of valuations on constructible functions (Lemma 12 in \cite{baryshnikov_hadwigers_2013}). By identifying a definable set $S$ with its characteristic function $\mathbf{1}_S$, any continuous, Euclidean-invariant valuation $v$ on sets induces a valuation $V$ on constructible functions. Baryshnikov establishes that $V$ is an integral against the intrinsic volumes; restricting back to the characteristic function $\mathbf{1}_S$ yields $v(S) = \sum c_i \mu_i(S)$.
\end{proof}

\subsection{Monotonicity Theorem}

To end this section, we prove a corollary of Hadwiger's Theorem that shows $\Phi$ is unique under certain conditions. However, since $\Phi$ satisfies Hadwiger's conditions, this shows that reframing the intrinsic volumes as polynomials gives us strict monotonicity as a necessary consequence. 

\begin{theorem}[Classification and Monotonicity]\label{classification}Let $\Psi: \mathcal{U}^b \to \mathbb{R}[t]$ be a multiplicative, similarity-equivariant valuation (note this implies Euclidean invariance). Suppose $\Psi$ is:\begin{itemize}
    \item Conormal Continuous: $\Psi$ is continuous with respect to the flat topology (as defined in \cite{baryshnikov_hadwigers_2013}).
    \item Positive on the unit interval: The leading coefficient of $\Psi([0,1))$ is positive.\footnote{Without this condition, alternating-sign behaviors might be possible.}
\end{itemize}
Then $\Psi$ is strictly monotone with respect to the lexicographical order. Specifically, there exists a unique $\alpha > 0$ such that $\Psi(A)(t) = \Phi(A)(\alpha t)$.
\end{theorem}

\begin{proof}
First, consider the behavior on the unit interval $I = [0,1)$. By the Scale Factor Lemma, $\Psi(I)(t) = \alpha t$ for some constant $\alpha$. The positivity hypothesis ensures $\alpha > 0$.

Next, consider the coefficient projection maps defined by $\nu_k(A) = [\Psi(A)]_k$, where $[\cdot]_k$ denotes the coefficient of $t^k$. Since $\Psi$ is a continuous, invariant valuation, each component $\nu_k$ is a real-valued continuous invariant valuation. By Hadwiger's Theorem for definable sets, each $\nu_k$ is a linear combination of intrinsic volumes. 

However, the similarity-equivariance of $\Psi$ imposes a strict constraint. For any scalar $\lambda > 0$, we have $\Psi(\lambda A)(t) = \Psi(A)(\lambda t)$, which implies:
$$ \nu_k(\lambda A) = \lambda^k \nu_k(A). $$
Since the intrinsic volume $\mu_j$ is homogeneous of degree $j$, the only term in the linear expansion of $\nu_k$ that scales with degree $k$ is $\mu_k$ itself. Thus, for each $k$, we must have $\nu_k(A) = c_k \mu_k(A)$ for some constant $c_k$. Summing these components, $\Psi(A)(t) = \sum_{k} c_k \mu_k(A) t^k$.

To determine the constants $c_k$, we invoke multiplicativity on the unit hypercubes. Let $I^m = [0,1)^m$\footnote{We use half-open intervals to ensure o-minimal Euler characteristic zero. This is critical to ensure strict monotonicity}. We have:

$$ \Psi(I^m) = (\Psi(I))^m = (\alpha t)^m = \alpha^m t^m. $$

Simultaneously, evaluating the sum formula on the hypercube (where $\mu_m(I^m)=1$ and $\mu_j(I^m)=0$ for $j \neq m$) yields $\Psi(I^m) = c_m t^m$. Comparing coefficients, we find $c_m = \alpha^m$. For $m=0$, note that the empty tuple set, $\R^0=\{()\}$ has Euler characteristic $1$. Moreover, multiplicativity implies we must have $\Psi(\R^0) = 1$; thus, $c_0 = \alpha^0 = 1$.

Substituting these constants back into the sum yields:
$$ \Psi(A)(t) = \sum_{k} \alpha^k \mu_k(A) t^k = \sum_{k} \mu_k(A) (\alpha t)^k = \Phi(A)(\alpha t). $$
Since $\alpha > 0$ and $\Phi$ is strictly monotone (Theorem 3), the scaling by $\alpha$ preserves the sign of the leading coefficient. Thus, $\Psi$ is strictly monotone.
\end{proof}

The intrinsic volumes already encode much of the geometric information about tame sets; Theorem \ref{classification} illustrates that they also retain enough information about the size of sets to enforce strict monotonicity under certain symmetries. 

\subsection{Algebraic Characterization of $\Phi$} To create a natural algebraic characterization of Theorem \ref{classification}, we extend our domain from sets to functions. Let $\mathcal{C}(\R^\infty)$ be the vector space of constructible functions with compact support on $\mathbb{R}^\infty$ (the direct limit under canonical embedding). We equip $\mathcal{C}(\R^\infty)$ with a ring structure via pointwise addition and the external product multiplication: $(f\boxtimes g)(x,y)=f(x)g(y)$ (corresponding to the Cartesian product of sets). We order this ring by the standard pointwise partial order: $f \le g \iff f(x) \leq g(x)$ for all $x$ (corresponding to the subset partial ordering). Since $\Phi$ is a valuation on bounded definable sets, it extends uniquely to a linear map on $\mathcal{C}(\R^\infty)$ via integration.
\begin{definition}[Extension to Functions]For a constructible function $f \in \mathcal{C}(\R^\infty)$, we define the intrinsic volume polynomial $\Phi(f) \in \mathbb{R}[t]$ by:

$$\Phi(f)(t) := \sum_{k \geq 0} \left( \int_{\R^\infty} f \, d\mu_k \right) t^k$$

where integration is with respect to the $k$-th intrinsic volume $\mu_k$ as defined in \cite{wright_hadwiger_2011}.\end{definition} 

This extension allows us to algebraically condense the geometric properties of Theorem \ref{classification} into an algebraic reformulation of Hadwiger's Characterization Theorem as follows:
\begin{corollary}[Algebraic Characterization] The intrinsic volume polynomial $\Phi$ is, up to scaling of the variable $t$ by a positive constant, the unique conormal continuous, similarity-equivariant homomorphism of ordered rings from $\mathcal{C}(\R^\infty) \to \mathbb{R}[t]$. \end{corollary}

Recall that ``similarity-equivariant" means that for any combination $g$ of rigid motions and scalings by a positive constant $\alpha$, $\Phi(g\cdot f)(t)=\Phi(f)(\alpha t)$. 

\begin{remark}[Normalization]The ambiguity of the variable scaling corresponds to the choice of unit length. If we impose the normalization condition $\Phi(\mathbf{1}_{[0,1)}) = t$, the map is uniquely determined as the specific intrinsic volume polynomial defined above.\end{remark}

\section{Nonstandard Analysis and Numerosity}

Having established that the intrinsic volume polynomial is strictly monotone, we explore how to model it through numerosity theory, which is concerned with modeling ``whole is greater than the part" behavior of point sets. Our goal in this section is to construct a numerosity function $\mathfrak n$ that approximates $\Phi$ in the sense that $$\mathfrak{n}(A) \approx \Phi(A)(\omega),$$ where $\approx$ means infinitely close and $\omega$ is an infinite hyperinteger. 

This inquiry leads us to form a bridge between concrete geometry and the abstract model-theoretic world of Nonstandard Analysis (NSA). Numerosity theory is a branch of NSA initialized by Benci and Di Nasso in \cite{benci_elementary_2014} and several other papers on the subject. A more general treatise on NSA can be found in \cite{cutland_eightfold_2016}. 

\subsection{Comparing Numerosities to the Intrinsic Volume Polynomial}

As a notational consideration before we continue, let us consider that numerosities are defined in a fixed dimension $d$, but we are concerned with products over different dimensions. Moreover, we will be considering intrinsic valuations, for which the ambient dimension has no effect on the measurement of a set. In this vein, adopt the following convention.  

\begin{remark}[Canonical Embeddings]
For the remainder of this paper, we fix a sufficiently large ambient dimension $d$. For any $p < d$, we identify $\mathbb{R}^p$ with the subspace $\mathbb{R}^p \times \{0\}^{d-p} \subset \mathbb{R}^d$. Accordingly, any definable subset $A \subset \mathbb{R}^p$ is identified with its canonical embedding $A \times \{0\}^{d-p} \subset \mathbb{R}^d$. This allows us to treat the numerosity $\mathfrak{n}$ as acting on definable sets of varying dimensions without explicit reference to the embedding map.
\end{remark}

We take the formal definition and elementary properties of a numerosity function from \cite{benci_elementary_2014}: 
\begin{definition}[Elementary Numerosity]\label{numerosity def}
An \emph{elementary numerosity} (or just \textit{numerosity}) on a set $\Omega$ is a function
$$\mathfrak{n}:\mathcal{P}(\Omega)\to [0,+\infty)_\mathbb{F}$$
defined for all subsets of $\Omega$,
taking values in the non-negative part of a field $\mathbb{F}\supset \mathbb R$,
and such that the following two conditions are satisfied:

\smallskip
\begin{enumerate}
\item
$\mathfrak{n}(\{x\})=1$ for every point $x\in\Omega$\,;

\smallskip
\item
$\mathfrak{n}(A\cup B)=\mathfrak{n}(A)+\mathfrak{n}(B)$ whenever $A$ and $B$ are disjoint.
\end{enumerate}
\end{definition}

\begin{proposition}\label{numerosity prop}
Let $\mathfrak{n}$ be an elementary numerosity. Then:

\smallskip
\begin{enumerate}
\item
$\mathfrak{n}(A)=0$ if and only if $A=\emptyset$;

\smallskip
\item
If $A\subsetneq B$, then $\mathfrak{n}(A)<\mathfrak{n}(B)$.

\smallskip
\item
If $F$ is a finite set of cardinality $n$, then $\mathfrak{n}(F)=n$;
\end{enumerate}
\end{proposition}

Note that several (but not all) of the same properties enjoyed by $\Phi$ are held by $\mathfrak{n}$. They describe very similar structures on point sets. The major differences include: 
\begin{enumerate}
    \item The numerosity $\mathfrak{n}$ is defined on the power set of $\mathbb{R}^d$, not just definable sets. 
    \item $\mathfrak{n}$ takes values in a field containing $\mathbb R$ while $\Phi$ takes values in $\mathbb{R}[t]$, which contains $\mathbb{R}$, but is not a field (though it may be extended to one). 
    \item In contrast to $\Phi$, the numerosity $\mathfrak{n}$ is rarely constructed in a manner as to be actively computed precisely. 
    \item While some definitions of a numerosity require it to respect Cartesian products (similar to $\Phi$), not all do. We do not take this as a preliminary here. 
\end{enumerate} 

In remainder of this section, we use the methods from \cite{benci_elementary_2014} as a guide to construct a hyperfinite sample $F \subset {}^*\mathbb{R}^d$ such that counting the points of $F$ lying inside a set $A$ yields almost exactly the value of the intrinsic volume polynomial (within infinitesimal error). This implies that $\Phi(A)$ can be modeled not just as a formal polynomial, but as a true asymptotic counting function. Furthermore, $\Phi$ can thus be ``extended" (with small error) to measure all subsets of $\mathbb R^d$. 

\subsection{Integer Approximation of the Polynomial}

We first prove that for any finite collection of tame sets, there exists a finite collection of points such that the point count approximates the value of the intrinsic volume polynomial. For this count to approximate the polynomial, we need the values of that polynomial to be nearly integers. Lemma \ref{small int parts} proves approximation is always possible, and Lemma \ref{nonstandard} achieves it.

\begin{lemma}\label{small int parts}
Let $||\alpha||$ denote the distance from $\alpha\in \mathbb{R}$ to the nearest integer. Fix $0<\epsilon<1$ and $K\in \N$. For a finite collection of polynomials $\{p_i(t)\}_i$ with integer constant terms ($p_i(0)\in \mathbb{Z}$), there exist infinitely many $N\in \mathbb{N}$ such that $||p_i(N)||<\epsilon$ for all $i$ and $N\equiv 0 \mod K$.
\end{lemma}

\begin{proof}
We proceed by induction on $k$, the number of polynomials. For $k=1$, either $p_1$ has only rational coefficients or it has at least one irrational coefficient. If it has an irrational one, Weyl's criterion shows that the sequence $p_1(Kn)$ is equidistributed modulo 1 \cite{stein_fourier_2003}, and thus there are infinitely many $N=Kn$ for which $||p_1(N)||<\epsilon$ and $N\equiv 0 \mod K$. If $p_1$ has only rational coefficients, we can simply choose $N$ to be a multiple of the common denominator of the coefficients and $K$ to get $||p_1(N)||=0$.

Now suppose the theorem holds for some fixed $k$. Consider a list $p_0, p_1, \dots, p_k$ of length $k+1$. If the set $\{p_0(Kt), \dots, p_k(Kt), 1, t, t^2, \dots\}$ is $\mathbb{Q}$-linearly independent, then for all $h\in \mathbb{Z}^{k+1}\setminus \{0\}$, the polynomial $q(t)=\sum_{i=0}^k h_i p_i(Kt)$ has an irrational coefficient and is thus equidistributed modulo 1. By Weyl's criterion for multidimensional equidistribution, there exist infinitely many $N=Kn$ for which $||p_i(N)||<\epsilon$ for all $i=0,\dots,k$.

If, on the other hand, the collection is $\mathbb{Q}$-linearly dependent, Weyl's criterion does not apply directly. However, we can write:
$$p_0(Kt) = q_1 p_1(Kt) + \cdots + q_k p_k(Kt) + R(Kt)$$
where $q_i \in \mathbb{Q}$ and $R(t) \in \mathbb{Q}[t]$. Let $Q$ be the product of all denominators appearing in the $q_i$ and the coefficients of $R$. Then for any natural number $n$, the integer distance satisfies:
$$||p_0(QKn)|| = \left|\left| \sum_{i=1}^k q_i {p}_i(QKn) \right|\right| =\left|\left| \sum_{i=1}^k q_i \tilde{p}_i(QKn) \right|\right|,$$
where $\tilde{p}_i$ denotes $p_i$ with all rational-coefficient terms removed. The second equality holds because the difference between $p_i(QKn)$ and $\tilde{p}_i(QKn)$ is always an integer. Recall that each $p_i$ has integer constant term, so the constant term $\tilde{p}_i(0)$ must equal zero. Hence, we can apply the inductive hypothesis to the $k$ polynomials: $$\frac{1}{Q}\tilde{p}_1(Qt), \dots, \frac{1}{Q}\tilde{p}_k(Qt).$$

Choose a target error $\delta = \epsilon / (Q + \sum_{i=1}^k |q_i|Q)$. The hypothesis guarantees an $M=Kn$ such that $||\frac{1}{Q}\tilde{p}_i(QM)|| < \delta$ for all $i=1,\dots,k$. This implies that $\frac{1}{Q}\tilde{p}_i(QM)$ is within $\delta$ of some integer $J_i$. Multiplying by $Q$, we see that $\tilde{p}_i(QM)$ is within $Q\delta$ of the integer multiple $Q J_i$.

For the basis polynomials $p_i$ ($i=1\dots k$), since $\tilde{p}_i(QM)$ is close to the integer $Q J_i$, we have:
$$ ||p_i(QM)|| = ||\tilde{p}_i(QM)|| < Q\delta < \epsilon.$$

For the dependent polynomial $p_0$, we note that for $i>0$ we have $\frac{1}{Q}\tilde{p}_i(QM) =  J_i+\eta_i$ for some $|\eta_i|<\delta$. Therefore $q_i\tilde{p}_i(QM) =  q_iQJ_i+q_iQ\eta_i$. In the sum, $\sum_i q_i\tilde{p}_i(QM) =  \sum_iq_iQJ_i+q_iQ\eta_i$. Since $Q$ is a common denominator for $q_i$, it follows $q_iQJ_i$ is an integer, so is ignored by $||\cdot||$. Combining this with our earlier equation, $$||p_0(QM)||=\bigg|\bigg|\sum_{i=1}^k q_i\tilde{p}_i(QM)\bigg|\bigg| =  \bigg|\bigg| \sum_{i=1}^kq_iQ\eta_i\bigg|\bigg|\leq \sum_{i=1}^k |q_iQ\eta_i| \leq\delta \sum_{i=1}^k |q_iQ|<\epsilon.$$

Therefore, at $N=QM$, $||p_i(N)||<\epsilon$ for all $i=0,\dots,k$. Because $N=QM=QKn$, $N \equiv 0 \mod K$. By induction, the proof is complete.\end{proof}

\begin{lemma}\label{nonstandard}
Fix $\epsilon=1/m$ for integer $m>0$ and $K\in \N$. Let $A_1,\dots, A_v \in \mathcal{U}^b$, and let $x_1,\dots,x_k\in \mathbb{R}^d$ be a finite set of points. Then there exists a finite set $\lambda \subset \mathbb{R}^d$ such that $\#(\lambda\cap [0,1))\equiv 0 \mod K$,  $\{x_1,\dots,x_k\} \subseteq \lambda$ and for all $i=1,\dots,v$:
$$|\#(\lambda \cap A_i) - \Phi(A_i)(\#(\lambda\cap [0,1)))|<\epsilon.$$
\end{lemma}

\begin{proof}
Let $A = A_1 \cup \cdots \cup A_v \cup \{x_1, \dots, x_k\}$.
Define $B_1, \dots, B_u$ as the partition of $[0,1)$ generated by the intersections $A_i \cap [0,1)$.
Similarly, define $C_1, \dots, C_w$ as the partition of $A \setminus [0,1)$ generated by the sets $A_i \setminus [0,1)$. Note that $B_j, C_j \in \mathcal{U}^b$.

We apply Lemma \ref{small int parts} to the collection of intrinsic volume polynomials $\{\Phi(B_j), \Phi(C_j)\}$. Note that $\Phi(S)(0) = \mu_0(S) = \chi(S) \in \mathbb{Z}$, so the integer constant term condition is satisfied.

Let $\lambda_0' = \{x_1, \dots, x_k\} \cup \bigcup \{B_j : B_j \text{ is finite}\}$. Choose $N \in K\mathbb{N}$ sufficiently large such that:
\begin{itemize}
    \item $N > \#(\lambda_0')$,
    \item $\Phi(S)(N) > k+1$ for all non-constant polynomials in our set,
    \item $||\Phi(B_j)(N)|| < \frac{\epsilon}{2\max(u,w)}$ and $||\Phi(C_j)(N)|| < \frac{\epsilon}{2\max(u,w)}$ for all $j$.
\end{itemize}

We construct $\lambda_0$ inside the unit interval. For each infinite $B_j$, pick exactly $[\Phi(B_j)(N)]$ points in $B_j$ and add them to $\lambda_0'$. For finite $B_j$, $\Phi(B_j)(N) = \#(B_j)$, so the count is already exact.
Since $\{B_j\}$ partitions $[0,1)$, we verify the total count:
\begin{align*}
\sum_{j=1}^u \#(\lambda_0 \cap B_j) &= \sum_{j=1}^u [\Phi(B_j)(N)] \\
&= \sum_{j=1}^u (\Phi(B_j)(N) \pm ||\Phi(B_j)(N)||) \\
&= \Phi([0,1))(N) + \sum \pm \text{error} \\
&= N + \delta.
\end{align*}
Since the total error is strictly less than 1 and the sum must be an integer, the error is exactly 0. Thus, $\#(\lambda_0 \cap [0,1)) = N$.

We construct $\lambda_1$ outside the unit interval. Let $\lambda = \lambda_0 \cup \lambda_1$. For each $C_j$, choose exactly $[\Phi(C_j)(N)] - \#(C_j \cap \lambda_0)$ new points. This is possible because for infinite sets the target number is large ($>k$), and for finite sets the count matches the Euler characteristic exactly.

By construction, for any set $S$ in our partition $\{B_j, C_j\}$, we have $|\#(\lambda \cap S) - \Phi(S)(N)| < \frac{\epsilon}{2\max(u,w)}$.
Finally, since each $A_i$ is a disjoint union of some sub-collection of these partition sets, the errors add linearly. Thus:
$$|\#(\lambda \cap A_i) - \Phi(A_i)(N)| < \epsilon.$$
Since $N = \#(\lambda \cap [0,1))$ was given to be a multiple of $K$, the proof is complete.
\end{proof}

\subsection{The Existence of Numerosity}

We now use the integer approximation results to construct the hyperfinite sample $F \subset {}^*\mathbb{R}^d$. The previous lemma asserts that for any finite configuration of sets, there is a finite sample that suitably models the size of each set. In the nonstandard extension, this implies the existence of a hyperfinite sample that is suitable for \textit{all} sets simultaneously. Furthermore, this sample's intersection with the unit interval has cardinality $\omega\in {}^*\N$, and $\omega$ is \textit{hyper-divisible} as we define below. This is a consequence of us choosing the above theorems to work with multiples of a given fixed $K\in \N$. 

\begin{definition}
    We call a hyperinteger \textit{hyper-divisible} if it is divisible by every standard natural number. 
\end{definition}

\begin{theorem}\label{nonstandard theorem}
In any model of nonstandard analysis satisfying the property of $\kappa$-enlargement where $\kappa > 2^{\#(\mathbb{R}^d)}$, there exists a numerosity function $\mathfrak{n}: \mathcal{P}(\mathbb{R}^d) \to {}^*\mathbb{N}$ and a hyper-divisible hyperinteger $\omega \in {}^*\mathbb{N}$ such that for any bounded definable set $A \in \mathcal{U}^b$:
$$\mathfrak{n}(A) \approx \Phi(A)(\omega).$$
Moreover, for these sets, $\mathfrak{n}(A) = [\Phi(A)(\omega)]$ (the nearest integer to $\Phi(A)(\omega)$). Finally, if $\Phi(A)$ has only rational coefficients, then $\mathfrak{n}(A) = \Phi(A)(\omega)$ exactly. 
\end{theorem}

\begin{proof}
Let $\Lambda$ be the set of all finite subsets of $\mathbb{R}^d$. We construct a family of subsets of $\Lambda$ to apply the enlargement property.
\begin{itemize}
    \item For $x \in \mathbb{R}^d$, let $S_x = \{ \lambda \in \Lambda : x \in \lambda \}$.
    \item For $A \in \mathcal{U}^b$ and $m \in \mathbb{N}$, let $\Gamma(A,m) = \{ \lambda \in \Lambda : |\#(\lambda \cap A) - \Phi(A)(\#(\lambda \cap [0,1)))| < 1/m \}$.
    \item For $k\in \N$, let $\Delta_k = \{ \lambda \in \Lambda : \#(\lambda \cap [0,1)) \equiv 0 \pmod k \}$.
\end{itemize}
Let $\mathcal{G} = \{S_x : x \in \mathbb{R}^d\} \cup \{\Gamma(A,m) : A \in \mathcal{U}^b, m \in \mathbb{N}\}\cup\{\Delta_k:k\in \N\}$.
Lemma \ref{nonstandard} guarantees that this family $\mathcal{G}$ has the \textbf{finite intersection property}. For any finite sub-collection, we can choose $K$ to be the product of all $k\in \N$ determining the $\Delta_k$'s, then pick a sample $\lambda$ with $\#(\lambda \cap [0,1))=K$ such that it contains the points necessary for the $S_x$'s and the $\Gamma(A,m)$'s. 

By $\kappa$-enlargement, there exists a hyperfinite set $F$ in the intersection of the star-transform of all sets in $\mathcal{G}$:
$$F \in \bigcap_{G \in \mathcal{G}} {}^*G.$$
This hyperfinite sample $F$ satisfies:
\begin{enumerate}
    \item $\mathbb{R}^d \subseteq F$ (every standard point is in $F$).
    \item For every $A \in \mathcal{U}^b$ and standard $m$, $|{}^*\#(F \cap {}^*A) - \Phi(A)({}^*\#(F \cap {}^*[0,1)))| < 1/m$.
\end{enumerate}
Define $\omega := {}^*\#(F \cap {}^*[0,1))$. Since $F\in {}^*\Delta_k$ for all $k\in\N$, it follows that $\omega$ is hyper-divisible: all standard natural numbers evenly divide $\omega$.  
Define the numerosity $\mathfrak{n}(A) := {}^*\#(F \cap {}^*A)$. First, we verify that $\mathfrak{n}$ is a valid numerosity on $\mathcal{P}(\mathbb{R}^d)$. By the transfer of finite additivity of cardinality, $\mathfrak{n}$ is finitely additive. Furthermore, condition (1) ensures that for any standard point $x \in \mathbb{R}^d$, $\{x\} \subset F$, implying $\mathfrak{n}(\{x\}) = {}^*\#(\{x\}) = 1$. Thus $\mathfrak{n}$ is a numerosity. 

From property (2), the difference between $\mathfrak{n}(A)$ and $\Phi(A)(\omega)$ is smaller than any standard positive real, hence they are infinitely close.
Since $\mathfrak{n}(A)$ is a hyperinteger \footnote{Note $\Phi(A)(0)\in\Z$}, it must be the unique hyperinteger closest to $\Phi(A)(\omega)$. Since $\omega$ is hyper-divisible, when the coefficients of $\Phi(A)$ are rational, $\Phi(A)(\omega)$ is already a hyperinteger, and thus $\mathfrak{n}(A)=\Phi(A)(\omega)$. 
\end{proof}

Note that the set $F$ is specifically designed to approximate $\Phi$. It is not a standard lattice; rather, it possesses varying densities reflecting the coefficients of $\Phi$.

This proves that $\Phi$ approximates $\mathfrak{n}$ infinitely closely in general, and they are precisely equal in the rational coefficient case. As a bonus, we get to use a number $\omega$ with the strange but attractive property that is divisible by every standard integer. 

\subsection{Properties of the Numerosity}

Theorem \ref{nonstandard theorem} is profound in that it shows that a numerosity can be chosen such that it mirrors the intriguing characteristics of $\Phi$, at least on suitably tame sets. But the numerosity $\mathfrak{n}$ is defined on \textit{all} subsets of $\mathbb{R}^d$. It is therefore both well-behaved enough to be explicitly computable for tame sets and general enough to have a definite value even on arbitrary pathological sets. 

As an example of using the computation, the following corollary shows that $\mathfrak{n}$, when suitably normalized, recovers standard Hausdorff measure information. Specifically, it can recover the dimensionality and  measure of a set.

\begin{corollary}
For tame sets $A\in \mathcal{U}^b$, the normalized numerosity recovers the Hausdorff measure:
$$st\left(\frac{\mathfrak{n}(A)}{\omega^{\dim(A)}}\right) = \mathcal{H}^{\dim(A)}(A).$$
\end{corollary}

\begin{proof}
Let $d = \dim(A)$. By the theorem, $\mathfrak{n}(A) \approx \Phi(A)(\omega) = \mu_d(A)\omega^d + \mu_{d-1}(A)\omega^{d-1} + \dots + \mu_0(A)$.
Dividing by $\omega^d$:
$$\frac{\mathfrak{n}(A)}{\omega^d} \approx \mu_d(A) + \frac{\mu_{d-1}(A)}{\omega} + \dots$$
Taking the standard part eliminates the lower-order terms (which are infinitesimal), leaving $\mu_d(A)$. As noted in the proof of Theorem \ref{Strict Monotonicity}, $\mu_d(A)$ corresponds to the $d$-dimensional Hausdorff measure $\mathcal{H}^d(A)$.
\end{proof}

In our final proof, we show that the numerosity inherits the nice properties of $\Phi$ when applied to bounded tame sets. As noted in Definition \ref{numerosity def} and Proposition \ref{numerosity prop}, additivity (i.e. inclusion-exclusion), extension of cardinality, and strict monotonicity are all enjoyed by the numerosity $\mathfrak{n}$ we have constructed. These properties are satisfied on the entire $\mathcal{P}(\mathbb{R}^d)$. We add that when restricted to tame sets, $\mathfrak{n}$ also satisfies: 

\begin{theorem}[Properties of the Intrinsic Volume Numerosity]
The numerosity function $\mathfrak{n}: \mathcal{P}(\mathbb{R}^d) \to {}^*\mathbb{N}$ inherits the structural symmetries of $\Phi$. Specifically:
\begin{enumerate}
\item \textbf{Euclidean Invariance:} If $\phi$ is a rigid motion, then $\mathfrak{n}(A) = \mathfrak{n}(\phi(A))$ for definable $A$. 
\item \textbf{Rational Multiplicativity:} If $\Phi(A)$ and $\Phi(B)$ have rational coefficients and $A\times B\subset \R^d$, then $$\mathfrak{n}(A\times B)=\mathfrak{n}(A)\cdot \mathfrak{n}(B).$$ 
\item \textbf{General Near-Multiplicativity:} $$\frac{\mathfrak{n}(A \times B)}{\mathfrak{n}(A) \cdot \mathfrak{n}(B)}\approx 1$$ for any nonempty definable $A$ and $B$ such that $A\times B\subset \R^d$. Note that this does not imply they are infinitely close; we leave open the possibility that $\mathfrak{n}(A\times B)\not\approx\mathfrak{n}(A)\cdot \mathfrak{n}(B)$ for $\Phi(A)\not\in\Q[t]$. 
\end{enumerate}
\end{theorem}

\begin{proof}
    Let $A$ and $B$ be definable subsets of $\mathbb R^d$ and let $\phi$ be a rigid motion. We have $$\mathfrak{n}(A)=[\Phi(A)(\omega)]=[\Phi(\phi(A))(\omega)]=\mathfrak{n}(\phi(A)).$$ Thus Euclidean Invariance (1) is proven. 
    
    Rational Multiplicativity (2) follows immediately  from the multiplicativity of $\Phi$ and the fact that $\Phi(A)(\omega)=\mathfrak n(A)$ for $\Phi(A)\in \mathbb{Q}[t]$.
    
    For Near-Multiplicativity (3), notice that by our construction of $\omega$, $\Phi(A)(\omega)$ is infinitely close to a hyperinteger, as is $\Phi(B)(\omega)$. Letting $\Phi(A)(\omega)=k+\epsilon$ and $\Phi(B)(\omega)=m+\delta$, where $k,m\in {}^*\mathbb N$ and $\epsilon$ and $\delta$ are infinitesimal, we have \begin{align*}
        \mathfrak{n}(A\times B)&\approx \Phi(A\times B)(\omega)\\
        &=\Phi(A)(\omega)\cdot \Phi(B)(\omega)\\
        &=(k+\epsilon)(m+\delta)\\
        &=km+k\delta+m\epsilon+\epsilon\delta\\
        &\approx km +k\delta+m\epsilon
    \end{align*}
    Therefore, $\mathfrak{n}(A\times B)\approx km +k\delta+m\epsilon.$ Divide both sides by $\mathfrak n (A)\cdot \mathfrak{n}(B)=km\geq 1$ to get $$\frac{\mathfrak{n}(A\times B)}{\mathfrak n (A)\cdot \mathfrak{n}(B)}\approx \frac{km +k\delta+m\epsilon}{km}=1+\frac\delta m+\frac \epsilon k\approx 1.$$
\end{proof}

\section{AI Declaration}

Declaration of generative AI and AI-assisted technologies in the manuscript preparation process: During the preparation of this work the author used Gemini AI in order to aid writing style and assist in proofs. After using this tool, the author reviewed and edited the content as needed and takes full responsibility for the content of the published article.

\bibliographystyle{unsrt}
\bibliography{citations}

\end{document}